\documentclass[12pt,a4paper]{amsart}
\usepackage{amsfonts,color}
\usepackage{amsthm}
\usepackage{amsmath}
\usepackage{amscd}
\usepackage[latin2]{inputenc}
\usepackage{t1enc}
\usepackage[mathscr]{eucal}
\usepackage{indentfirst}
\usepackage{graphicx}
\usepackage{graphics}
\usepackage{pict2e}
\usepackage{epic}
\usepackage{url}
\usepackage{epstopdf}

\numberwithin{equation}{section}
\usepackage[margin=2.6cm]{geometry}

\theoremstyle{plain}
\newtheorem{Th}{Theorem}[section]
\newtheorem{Lemma}[Th]{Lemma}
\newtheorem{Cor}[Th]{Corollary}

 \theoremstyle{definition}
\newtheorem{Def}[Th]{Definition}

\newtheorem{Rem}[Th]{Remark}
\newtheorem{?}[Th]{Problem}
\newtheorem{Ex}[Th]{Example}

\begin{document}

\title{Counting degree-constrained subgraphs and orientations}

\author[M. Borb\'enyi]{M\'arton Borb\'enyi}

\address{ELTE: E\"{o}tv\"{o}s Lor\'{a}nd University \\ H-1117 Budapest
\\ P\'{a}zm\'{a}ny P\'{e}ter s\'{e}t\'{a}ny 1/C}

\email{marton.borbenyi@gmail.com}

\author[P. Csikv\'ari]{P\'{e}ter Csikv\'{a}ri}

\address{MTA-ELTE Geometric and Algebraic Combinatorics Research Group \& ELTE: E\"{o}tv\"{o}s Lor\'{a}nd University \\ Mathematics Institute, Department of Computer 
Science \\ H-1117 Budapest
\\ P\'{a}zm\'{a}ny P\'{e}ter s\'{e}t\'{a}ny 1/C}

\email{peter.csikvari@gmail.com}

\thanks{The first author is partially supported be the New National Excellence Program (\'UNKP), and when the project started he was partially supported by the EFOP program (EFOP-3.6.3-VEKOP-16-2017-00002). The second author  is supported by the 
 Marie Sk\l{}odowska-Curie Individual Fellowship grant no. 747430, and before that grant he was partially supported by the
Hungarian National Research, Development and Innovation Office, NKFIH grant K109684 and Slovenian-Hungarian grant NN114614, and by the ERC Consolidator Grant  648017.
}

 \subjclass[2010]{Primary: 05C30. Secondary: 05C31, 05C70}

 \keywords{Eulerian orientations, half graphs, gauge transformations} 

\begin{abstract} The goal of this  paper is to advertise the method of gauge transformations (aka holographic reduction, reparametrization) that is well-known in statistical physics and computer science, but less known in combinatorics. As an application of it we give a new proof of a theorem of A. Schrijver asserting that the number of Eulerian orientations of a $d$--regular graph on $n$ vertices with even $d$ is at least 
$\left(\frac{\binom{d}{d/2}}{2^{d/2}}\right)^n$. We also show that a $d$--regular graph with even $d$ has always at least as many Eulerian orientations as $(d/2)$--regular subgraphs. 
\end{abstract}

\maketitle

\section{Introduction}

In this paper we advertise a method that is well-known in statistical physics and computer science, but is less known is combinatorics. Roughly speaking this method enables one to transform a counting problem to another one that might be easier to analyse. In computer science this method was introduced by L. Valiant under the name holographic reduction \cite{valiant2008holographic,valiant2006accidental,valiant2002quantum,valiant2002expressiveness}. In statistical physics it was developed by M. Chertkov and V. Chernyak under the name gauge transformation \cite{chertkov2006loop2,chertkov2006loop1}. Wainwright, Jaakola, Willsky had a related idea under the name reparametrization \cite{wainwright2003tree}, but it is not easy to see the connection. In the different cases the scope was slightly different, L. Valiant used it as a reduction method for computational complexity of counting problems. This line of research was extended in a series of papers of Jin-Yi Cai and his coauthors, see Jin-Yi Cai's book \cite{cai2017complexity} and the papers \cite{cai2007holographic,cai2008basis,cai2007symmetric,cai2008holographic, cai2011holographic, cailu2008holographic} and references therein. M. Chertkov and V. Chernyak \cite{chertkov2006loop2,chertkov2006loop1} studied the so-called Bethe--approximation through gauge transformations. In this paper we mainly adopt the notations of gauge transformations, but we will give pointers to the other papers too and we also give the alternative names of certain concepts. 

This paper is written primarily for combinatorists, so the main objects of this paper will be orientations and subgraphs. From a mathematical point of view this method can be considered as an application of invariant theory to graph theory, but no knowledge of invariant theory is assumed in this paper. Below we collected these applications. In each case we give a theorem for regular graphs and its generalization for non-regular graphs. To keep the arguments simple we will mainly prove the theorems for regular graphs, and then we explain how to modify the arguments to be valid for non-regular graphs. We will also give various examples.
\bigskip

\subsection{Applications in graph theory} Recall that a graph is called Eulerian if all degrees are even. It is often assumed in the literature that an Eulerian graph $G$ is also connected too, but in this paper we do not require connectedness. An orientation of an Eulerian graph is called an Eulerian orientation if the in-degree and out-degree is the same at each vertex. Counting Eulerian orientations has triggered considerable interest both in combinatorics, computer science and statistical physics. Probably, the best known result is due to Lieb \cite{lieb2004residual} who determined the asymptotic number of Eulerian orientations of large grid graphs. Welsh \cite{welsh1999tutte} observed that for a $4$--regular graph the Tutte-polynomial evaluation $|T_G(0,-2)|$  is exactly the number of Eulerian orientations since nowhere-zero $Z_3$-flows and Eulerian orientations are in one-to-one correspondence for $4$--regular graphs. Mihail and Winkler \cite{mihail1992number} gave an efficient randomized algorithm to sample and approximately count Eulerian orientations. 

Our first result will be a new proof of a  lower bound on the number of Eulerian  orientations due to A. Schrijver. First we give it for non-regular graphs, then for regular graphs.

\begin{Th}[A. Schrijver \cite{schrijver1983bounds}] \label{Schrijver-non-regular}
Let $G$ be a  graph on $n$ vertices with degree sequence $d_1,d_2,\dots ,d_n$, where $d_k$ are even for all $k$. Let $\varepsilon(G)$ denote the number of Eulerian orientations of the graph $G$. Then
$$\varepsilon(G)\geq \prod_{k=1}^n\frac{\binom{d_k}{d_k/2}}{2^{d_k/2}}.$$
\end{Th}

\begin{Cor}[A. Schrijver \cite{schrijver1983bounds}] \label{Schrijver-regular}
Let $G$ be a $d$--regular graph on $n$ vertices, where $d$ is even. Let $\varepsilon(G)$ denote the number of Eulerian orientations of the graph $G$. Then
$$\varepsilon(G)\geq \left(\frac{\binom{d}{d/2}}{2^{d/2}}\right)^n.$$
\end{Cor}

In our proof of Theorems~\ref{Schrijver-non-regular} and Corollary~\ref{Schrijver-regular} we will improve on the lower bounds by a multiplicative factor $2$. Practically, we will give a formula for the number of Eulerian orientations with only non-negative terms and two main terms corresponding exactly to Schrijver's lower bound.

Next we compare the number of Eulerian orientations with the number of certain subgraphs.

\begin{Def}
Let $G$ be an Eulerian graph. A graph $H$ is a half-graph of $G$ if it is a spanning subgraph of $G$, and $d_H(v)=d_G(v)/2$ for all vertex $v$.
\end{Def}

\begin{Th} \label{orientations and half-graphs, non-regular}
Let $G$ be an Eulerian graph. Let $\varepsilon(G)$ denote the number of Eulerian orientations of the graph $G$, and let $h(G)$ denote the number of half-graphs of $G$. Then $\varepsilon(G)\geq h(G)$. Equality holds if and only if $G$ is bipartite.
\end{Th}

\begin{Cor} \label{orientations and half-graphs, regular}
Let $G$ be a $(2k)$--regular graph. Then it has at least as many Eulerian orientations as $k$--regular subgraphs. Equality holds if and only if $G$ is bipartite.
\end{Cor}

Next we study random orientations.

\begin{Th} \label{sources and sinks}
Let $G$ be a connected $3$--regular graph on $n$ vertices. Let us choose an orientation $\mathcal{O}$ of $G$ uniformly at random, and let $n_+(\mathcal{O})$ be the number of vertices with out-degree $3$, and let $n_-(\mathcal{O})$ be the number of vertices with in-degree $3$. Then the probability that $n_+(\mathcal{O})-n_-(\mathcal{O})=k$ is exactly $\frac{\binom{n}{n/2-2k}}{2^{n-1}}$.
\end{Th}

\subsection{Subgraph counting polynomial}

The main object that we will study in this paper is the following multivariate graph polynomial. We will call it the subgraph counting polynomial.
First we introduce it for regular graphs, and then for non-regular graphs.

\begin{Def} Let $G$ be a $d$--regular graph. Then the subgraph counting polynomial of $G$ is defined as 
$$F_G(x_0,x_1,\dots ,x_d)=\sum_{A \subseteq E(G)}\prod_{v\in V(G)}x_{d_A(v)},$$
where $d_A(v)$ is the degree of the vertex $v$ in the subgraph $(V,A)$.
\end{Def}

\begin{Ex}
\begin{align*}
F_{K_4}(x_0,x_1,x_2,x_3)&=x_0^4 + 6x_0^2x_1^2 + 3x_1^4 + 12x_0x_1^2x_2  \\
&+ 12x_1^2x_2^2+4x_0x_2^3+ 3x_2^4 + 4x_1^3x_3  \\
&+ 12x_1x_2^2x_3 + 6x_2^2x_3^2 + x_3^4.
\end{align*}
\end{Ex}

This polynomial naturally encodes certain counting problems. For instance, $F_G(0,1,0,\dots ,0)$ simply counts the number of perfect matchings of the graph $G$. Invariant theory comes into the picture by the observation that $F_G(x_0,\dots,x_d)$ is invariant under some group actions. For instance, $F_G(0,0,1,0,0)=F_G\left(\frac{3}{2},0,-\frac{1}{2},0,\frac{3}{2}\right)$ for any $4$--regular graph $G$.
\medskip

The reason why we first introduced the subgraph counting polynomial of a regular graph $G$ is that for non-regular graphs the useful generalization is not the natural one. The natural one would be to keep the definition with $d$ being the maximum degree. The useful or correct generalization is to first introduce $d_v+1$ variables for each vertex $v$, namely, $x^v_0,x^v_1,\dots ,x^v_{d_v}$ and we denote by $\underline{x}$ the collection of all variables $x^v_k$ for all $v$ and $k$. Then we can define the multi-affine polynomial 
$$F_G(\underline{x})=\sum_{A \subseteq E(G)}\prod_{v\in V(G)}x^v_{d_A(v)},$$
where $d_A(v)$ is the degree of the vertex $v$ in the subgraph $(V,A)$.
\medskip

\subsection{Organization of the paper}

This paper is organized as follows. In the next section we review the concept of normal factor graph and gauge transformation. Then in Section~\ref{symmetric functions} we specialize the partition function of a normal factor graph to get the above defined subgraph counting polynomial. Then in Section~\ref{duality} we show how to express a summation to orientations by the subgraph counting polynomial. Utilizing this new observation we prove Theorems~\ref{Schrijver-non-regular}, \ref{Schrijver-regular}, \ref{orientations and half-graphs, non-regular} and \ref{orientations and half-graphs, regular} in Section~\ref{Eulerian orientations}.

\section{Normal factor graphs and gauge transformations} \label{normal factor graph}

In this section we first introduce the concept of a normal factor graph, and then the gauge transformations.

\begin{Def} A normal factor graph is a graph equipped with a function at each vertex: $\mathcal{H}=(V,E,(f_v)_{v\in V})$. At each edge $e$ there is a variable $x_e$ taking values from an alphabet $\mathcal{X}$. The partition function 
$$Z(\mathcal{H})=\sum_{\sigma\in \mathcal{X}^E}\prod_{v\in V}f_v(\sigma_{\partial v}),$$ 
where $\sigma_{\partial v}$ is the restriction of $\sigma$ to the the edges incident to the vertex $v$. 
If the alphabet $\mathcal{X}$ is of size $2$, then we call the normal factor graph a  binary normal factor graph.
\end{Def}

For instance, if $\mathcal{X}=\{0,1\}$ and 
$$f_v(\sigma_1,\dots \sigma_d)=\left\{ \begin{array}{cl} 1 & \mbox{if}\ \sum_{i=1}^d\sigma_i=1, \\ 0 & \mbox{otherwise}, \end{array} \right.$$
then $Z(\mathcal{H})$ is exactly the number of perfect matchings of the underlying graph. 
\bigskip

\begin{Rem} L. Valiant, Jin-Yi Cai and his coauthors call the problem of determination or approximation of $Z(\mathcal{H})$ the Holant problem and $Z(\mathcal{H})$ itself the Holant. The emphasis in their papers are somewhat different: they would like to reduce the computation of $Z(\mathcal{H})$ to counting perfect matchings in planar graphs. Generally, in the reduction planarity does not play any role, still it is important to keep the graph planar under the reductions since we can count the number of perfect matchings in planar graphs. They introduce the so-called matchgates that are related to gauge transformations. 
J. M. Landsberg, J. Morton and S. Norine \cite{landsberg2013holographic} showed that holographic reduction can be carried out without matchgates.

\end{Rem}

Let $\mathcal{H}=(V,E,(f_v)_{v\in V})$ be a normal factor graph with alphabet $\mathcal{X}$. We will show that is possible to introduce a new normal factor graph $\widehat{\mathcal{H}}$ on the same graph with new functions $\widehat{f_v}$ such that $Z(\widehat{\mathcal{H}})=Z(\mathcal{H})$. As we will see sometimes it will be more convenient to study the new normal factor graph $\widehat{\mathcal{H}}$.

Let $\mathcal{Y}$ be a new alphabet, and for each edge $(u,v)\in E$ let us introduce two new matrices, $G_{uv}$ and $G_{vu}$ of size $\mathcal{Y}\times \mathcal{X}$. The new variables will be denoted by $\tau$, the old ones by $\sigma$. We will denote by $\mathbb{G}$ the collection of the matrices $G_{uv}$. Let
$$\widehat{f_{\mathbb{G},v}}(\tau_{vu_1},\dots ,\tau_{vu_k})=\sum_{\sigma_{vu_1},\dots ,\sigma_{vu_k}}\left(\prod_{u_i \in N(v)}G_{vu_i}(\tau_{vu_i},\sigma_{vu_i})\right)f_v(\sigma_{vu_1},\dots ,\sigma_{vu_k}).$$
This way we defined the functions $\widehat{f_{\mathbb{G},v}}$ of $\widehat{\mathcal{H}}$.

\begin{Th}[M. Chertkov and V. Chernyak \cite{chertkov2006loop2,chertkov2006loop1}] If for each edge $(u,v)\in E$ we have $G^T_{uv}G_{vu}=\mathrm{Id}_{\mathcal{X}}$, then $Z(\widehat{\mathcal{H}})=Z(\mathcal{H})$. 
\end{Th}

\begin{proof}
Let us start to compute 
$Z(\widehat{\mathcal{H}})=\sum_{\tau\in \mathcal{Y}^E}\prod_{v\in V}\widehat{f_{\mathbb{G},v}}(\tau_{\partial v})$:
$$Z(\widehat{\mathcal{H}})=\sum_{\tau\in \mathcal{Y}^E}\prod_{v\in V}\left[\sum_{\sigma_{vu_1},\dots ,\sigma_{vu_k}}\left(\prod_{u_i \in N(v)}G_{vu_i}(\tau_{vu_i},\sigma_{vu_i})\right)f_v(\sigma_{vu_1},\dots ,\sigma_{vu_k})\right].$$
If we expand it will have terms $\prod_{v\in V}f_v(\sigma_{vu_1},\dots ,\sigma_{vu_k})$ with some coefficients. A priori it can occur that these terms are incompatible in the sense that $\sigma_{uv}\neq \sigma_{vu}$. As we will see that the role of the conditions on $G_{uv}$ is exactly to ensure that if there is an edge $(u,v)\in E$ with 
$\sigma_{uv}\neq \sigma_{vu}$, then the coefficient is $0$, and if all edges are compatible, then the coefficient is $1$. Indeed, the coefficient is
$$\sum_{\tau\in \mathcal{Y}^E}\prod_{v\in V}\prod_{u_i \in N(v)}G_{vu_i}(\tau_{vu_i},\sigma_{vu_i}).$$
Note that $\tau_{uv}=\tau_{vu}$ for each edge, and this variable appears only at the vertices $u$ and $v$, and nowhere else. Hence
$$\sum_{\tau\in \mathcal{Y}^E}\prod_{v\in V}\prod_{u_i \in N(v)}G_{vu_i}(\tau_{vu_i},\sigma_{vu_i})=\prod_{(u,v)\in E}\left(\sum_{\tau_{uv}}G_{uv}(\tau_{uv},\sigma_{uv})G_{vu}(\tau_{vu},\sigma_{vu})\right)=$$
$$=\prod_{(u,v)\in E}\left(\sum_{\tau_{uv}}G^T_{uv}(\sigma_{uv},\tau_{vu})G_{vu}(\tau_{vu},\sigma_{vu})\right)=\prod_{(u,v)\in E}(G^T_{uv}G_{vu})_{\sigma_{uv},\sigma_{vu}}=
\prod_{(u,v)\in E}(\mathrm{Id})_{\sigma_{uv},\sigma_{vu}}.$$
Hence this is only non-zero if $\sigma_{uv}=\sigma_{vu}$ for each edge $(u,v)\in E(G)$, and then this coefficient is $1$.
\end{proof}

Next we show what happens when we apply two gauge transformations to the same function $f$. Suppose that $\mathbb{G}'$ is another collection of matrices $G'_{uv}$ of size $\mathcal{Z}\times \mathcal{Y}$, where $\mathcal{Z}$ is new alphabet whose elements will be denoted by $\gamma$. We will denote by $\mathbb{G}'\mathbb{G}$ the matrices $G'_{uv}G_{uv}$. 

\begin{Th} \label{multiplication} We have
$$\widehat{\widehat{f_{\mathbb{G},v}}_{\mathbb{G}',v}}=\widehat{f_{\mathbb{G}'\mathbb{G},v}}.$$
\end{Th}

\begin{proof}
We have
$$\widehat{f_{\mathbb{G},v}}(\tau_{vu_1},\dots ,\tau_{vu_k})=\sum_{\sigma_{vu_1},\dots ,\sigma_{vu_k}}\left(\prod_{u_i \in N(v)}G_{vu_i}(\tau_{vu_i},\sigma_{vu_i})\right)f_v(\sigma_{vu_1},\dots ,\sigma_{vu_k})$$
and so
$$\widehat{\widehat{f_{\mathbb{G},v}}_{\mathbb{G}',v}}(\gamma_{vu_1},\dots ,\gamma_{vu_k})=\sum_{\tau_{vu_1},\dots ,\tau_{vu_k}}\left(\prod_{u_i \in N(v)}G'_{vu_i}(\gamma_{vu_i},\tau_{vu_i})\right)\widehat{f_{\mathbb{G},v}}(\tau_{vu_1},\dots ,\tau_{vu_k})=$$
$$\sum_{\tau_{vu_1},\dots ,\tau_{vu_k}}\left(\prod_{u_i \in N(v)}G'_{vu_i}(\gamma_{vu_i},\tau_{vu_i})\right)\sum_{\sigma_{vu_1},\dots ,\sigma_{vu_k}}\left(\prod_{u_i \in N(v)}G_{vu_i}(\tau_{vu_i},\sigma_{vu_i})\right)f_v(\sigma_{vu_1},\dots ,\sigma_{vu_k})=$$
$$=\sum_{\sigma_{vu_1},\dots ,\sigma_{vu_k}}f_v(\sigma_{vu_1},\dots ,\sigma_{vu_k})\sum_{\tau_{vu_1},\dots ,\tau_{vu_k}}\prod_{u_i \in N(v)}G'_{vu_i}(\gamma_{vu_i},\tau_{vu_i})G_{vu_i}(\tau_{vu_i},\sigma_{vu_i})=$$
$$=\sum_{\sigma_{vu_1},\dots ,\sigma_{vu_k}}f_v(\sigma_{vu_1},\dots ,\sigma_{vu_k})\prod_{u_i \in N(v)}\left(\sum_{\tau_{vu_i}}G'_{vu_i}(\gamma_{vu_i},\tau_{vu_i})G_{vu_i}(\tau_{vu_i},\sigma_{vu_i})\right)=$$
$$=\sum_{\sigma_{vu_1},\dots ,\sigma_{vu_k}}f_v(\sigma_{vu_1},\dots ,\sigma_{vu_k})\prod_{u_i \in N(v)}(G'G)_{vu_i}(\gamma_{vu_i},\sigma_{vu_i})=\widehat{f_{\mathbb{G}'\mathbb{G},v}}(\gamma_{vu_1},\dots ,\gamma_{vu_k}).$$
We are done.
\end{proof}

\section{Normal factor graphs with symmetric functions} \label{symmetric functions}

Let us consider the binary normal factor graph with functions
$$f_v(\sigma_1,\dots ,\sigma_d)=x^v_k\ \ \ \mbox{if}\ \ \ \sum_{i=1}^d\sigma_i=k$$
for every $v\in V(G)$. Then we immediately get back the definition of the subgraph counting polynomial $F_G(\underline{x})$. Jin-Yi Cai and his coauthors call this the Holant function with symmetric signature, see for instance \cite{cai2007symmetric}.

That is, we can regard $F_G(\underline{x})$ as the partition function $Z(\mathcal{H})$, or as a polynomial in variables $x^v_0,\dots ,x^v_{d_v}$.
In what follows we omit $v$ from the notation and we simply write $x_k$ instead of $x^v_k$. Furthermore, we simply write $d$ instead of $d_v$. 
Let us use the same gauge everywhere
$$G_t:=\left( \begin{array}{cc} \cos(t) & \sin(t) \\ -\sin(t) & \cos(t) \end{array}\right).$$ 
Then
$$\widehat{f_{G_t,v}}(\tau_{vu_1},\dots ,\tau_{vu_d})=\sum_{\sigma_{vu_1},\dots ,\sigma_{vu_d}}\left(\prod_{u_i \in N(v)}G_{vu_i}(\tau_{vu_i},\sigma_{vu_i})\right)f_v(\sigma_{vu_1},\dots ,\sigma_{vu_d})$$
or more conveniently,
$$\widehat{f_{G_t,v}}(\tau_{1},\dots ,\tau_{d})=\sum_{\sigma_{1},\dots ,\sigma_{d}}\left(\prod_{i=1}^dG(\tau_{i},\sigma_{i})\right)f_v(\sigma_{1},\dots ,\sigma_{d})$$
only depends on the value $\sum_{i=1}^d\tau_{i}$. Let
$$\widehat{a_r}(t)=\widehat{f_{G_t,v}}(\tau_{1},\dots ,\tau_{d})\ \ \ \mbox{if} \ \ \ \sum_{i=1}^d\tau_i=r.$$
Then
$$\widehat{a_r}(t)=\sum_{k=0}^d\left[\sum_{s=0}^r\binom{r}{s}\binom{d-r}{k-s}(-1)^{r-s}\cos(t)^{d-r-k+2s}\sin(t)^{r+k-2s}\right]x_k.$$
Indeed, we first choose $s$ places where we keep $1$'s (and we switch the remaining $r-s$ pieces of $1$'s to $0$), then we need to choose $k-s$ places where we switch the $0$ to $1$ to get exactly $k$ pieces of $1$'s. Then in $d-r-k+2s$ cases we kept the original value, and in $r+k-2s$ we switched it, the sign $r-s$ comes from switching $r-s$ pieces of $1$'s to $0$.

\subsection{The functions $\widehat{a_r}(t)$}

In this section we study the functions $\widehat{a_r}(t)$.

\begin{Lemma} \label{rotation matrix rows}
Let us introduce a new variable $x$, and a linear map $L$ such that $L(x^k)=x_k$. Then
$$\widehat{a_r}(t)=L\left((x\cos(t)-\sin(t))^r(x\sin(t)+\cos(t))^{d-r}\right).$$
\end{Lemma}

\begin{proof}
By the binomial theorem the coefficient of $x^k$ in $(x\cos(t)-\sin(t))^r(x\sin(t)+\cos(t))^{d-r}$ is
$$\sum_{s=0}^r\binom{r}{s}\cos(t)^s(-\sin(t))^{r-s}\cdot \binom{d-r}{k-s}\sin(t)^{k-s}\cos(t)^{d-r-k+s}$$
which is equal to
$$\sum_{s=0}^r\binom{r}{s}\binom{d-r}{k-s}(-1)^{r-s}\cos(t)^{d-r-k+2s}\sin(t)^{r+k-2s}.$$
This is exactly the coefficient of $x_k$ in $\widehat{a_r}(t)$.
\end{proof}

\begin{Rem}
Another way to phrase the above lemma is the following. If $\alpha_1,\dots ,\alpha_m,\lambda_1,\dots ,\lambda_m$ satisfy that
$$x_k=\sum_{j=1}^m\alpha_j\lambda_j^k,$$
then
$$\widehat{a_r}(t)=\sum_{j=1}^m\alpha_j(\lambda_j\cos(t)-\sin(t))^r(\lambda_j\sin(t)+\cos(t))^{d-r}.$$

\end{Rem}

\begin{Lemma} We have
$$\frac{d}{dt}\widehat{a_r}(t)=(d-r)\widehat{a_{r+1}}(t)-r\widehat{a_{r-1}}(t).$$
\end{Lemma}

\begin{proof}  This immediately follows from the previous lemma since
$$\frac{d}{dt}(c\cos(t)-\sin(t))^r(c\sin(t)+\cos(t))^{d-r}=$$
$$=(d-r)(c\cos(t)-\sin(t))^{r}(c\sin(t)+\cos(t))^{d-r-1}(c\cos(t)-\sin(t))+$$
$$+r(c\cos(t)-\sin(t))^{r-1}(c\sin(t)+\cos(t))^{d-r}(-c\sin(t)-\cos(t))=$$
$$(d-r)\widehat{a_{r+1}}(t)-r\widehat{a_{r-1}}(t).$$
\end{proof}

\subsection{The rotation matrices}

Let $\underline{\widehat{a}(t)}$ be the vector with entries $\widehat{a_r}(t)$, and similarly let $\underline{x}$ with entries $x_k$ for $k=0,\dots, d$. Since all functions $\widehat{a_r}(t)$ is linear in $x_0,x_1,\dots ,x_d$ we can simply introduce the matrix $\textbf{R}_t$ for which $\textbf{R}_t\underline{x}=\underline{\widehat{a}(t)}$.  So far we proved that for all graph $G$ and $t\in \mathbb{R}$ we have $F_G(\textbf{R}_t\underline{a})=F_G(\underline{a})$. Together with the following lemma we see that we landed in the field of invariant theory.

\begin{Lemma} For all $t_1,t_2$ we have $\textbf{R}_{t_1}\textbf{R}_{t_2}=\textbf{R}_{t_1+t_2}$.
\end{Lemma}

\begin{proof} This is clear from Theorem~\ref{multiplication} and the fact that
$$\left( \begin{array}{cc} \cos(t_1) & \sin(t_1) \\ -\sin(t_1) & \cos(t_1)\ \end{array}\right)\left( \begin{array}{cc} \cos(t_2) & \sin(t_2) \\ -\sin(t_2) & \cos(t_2)\ \end{array}\right)=\left( \begin{array}{cc} \cos(t_1+t_2) & \sin(t_1+t_2) \\ -\sin(t_1+t_2) & \cos(t_1+t_2)\ \end{array}\right).$$

\end{proof}

\begin{Ex} For $d=4$ the matrix $\textbf{R}_t$ is the following.

{\fontsize{7}{3}\selectfont{
$$\left(
\begin{array}{ccccc}
\cos(t)^4 & 4\sin(t)\cos(t)^3 & 6\sin(t)^2\cos(t)^2 & 4\sin(t)^3\cos(t) & \sin(t)^4\\
-\sin(t)\cos(t)^3 & -3\sin(t)^2\cos(t)^2+\cos(t)^4 & -3\sin(t)^3\cos(t)+3\sin(t)\cos(t)^3 & -\sin(t)^4+3\sin(t)^2\cos(t)^2 & \sin(t)^3\cos(t)\\
\sin(t)^2\cos(t)^2 & 2\sin(t)^3\cos(t)-2\sin(t)\cos(t)^3 & \sin(t)^4-4\sin(t)^2\cos(t)^2+\cos(t)^4 & -2\sin(t)^3\cos(t)+2\sin(t)\cos(t)^3 & \sin(t)^2\cos(t)^2\\
-\sin(t)^3\cos(t) & -\sin(t)^4+3\sin(t)^2\cos(t)^2 & 3\sin(t)^3\cos(t)-3\sin(t)\cos(t)^3 & -3\sin(t)^2\cos(t)^2+\cos(t)^4 & \sin(t)\cos(t)^3\\
\sin(t)^4 & -4\sin(t)^3\cos(t) & 6\sin(t)^2\cos(t)^2 & -4\sin(t)\cos(t)^3 & \cos(t)^4\\
\end{array}
\right)$$
}}

\end{Ex}

\begin{Ex} For $d=4$ and $t=\pi/4$ we have
$$\textbf{R}_{\pi/4}=
\frac{1}{4}\left(
\begin{array}{ccccc}
1 & 4 & 6 & 4 & 1 \\
-1 & -2 & 0 & 2 & 1 \\
1 & 0 & -2 & 0 & 1 \\
-1 & 2 & 0 & -2 & 1 \\
1 & -4 & 6 & -4  & 1
\end{array}
\right)
$$
In particular, for a $4$--regular graph $G$ we have
$$F_G(0,0,1,0,0)=F_G\left(\frac{3}{2},0,-\frac{1}{2},0,\frac{3}{2}\right).$$
\end{Ex}

\begin{Rem}
The matrices $\textbf{R}_{\pi/4}$ are studied under the name Krawtchouk matrices. For more details on these matrices see the paper \cite{feinsilver2005krawtchouk}.

\end{Rem}

\subsection{An algebraic point of view}

Let $R=\mathbb{C}[x_0,\dots ,x_d]$ and $\partial$ is a derivation: a map satisfying $\partial(a+b)=\partial(a)+\partial(b)$ and $\partial(ab)=b\partial(a)+a\partial(b)$ for every $a,b\in R$. In general if we know that $\partial (x_k)=f_k$, then
$$\partial P=\sum_{k=0}^df_k\frac{d}{dx_k}P.$$

Having a derivation $\partial$ we can consider its ring of coefficients, that is, its kernel:
$$R^{\partial}=\{a\in R\ |\ \partial(a)=0\}.$$
This is indeed a ring.

For our goals we consider the derivation for which
$$\partial(x_k)=(d-k)x_{k+1}-kx_{k-1}.$$

\begin{Th}
Let $\partial $ be defined by
$$\partial P=\sum_{k=0}^d((d-k)x_{k+1}-kx_{k-1})\frac{d}{dx_k}P.$$
Then for any $d$--regular graph $G$ we have $\partial F_G(x_0,\dots ,x_d)=0$. In other words,\\
$F_G(x_0,\dots ,x_d)\in R^{\partial}$.
\end{Th}

\begin{Rem}
Clearly, the theorem is motivated by the observation that $F_G(\widehat{a_0}(t),\dots ,\widehat{a_d}(t))$ is independent of $t$, and  for $\partial=\frac{d}{dt}$ we get the same relations for $\widehat{a_r}(t)$. This observation leads to an alternative proof of the above theorem.

\end{Rem}

\begin{proof}
We will simply use the definition $F_G(x_0,\dots ,x_d)=\sum_{A \subseteq E(G)}\prod_{v\in V(H)}x_{d_A(v)}$.
Let $e=(u,v)$ be an edge of $G$, and let $A$ be a subset of the edges of $G$ such that $e\notin A$. Let us introduce the notation 
$$T_{+}(A,v,e)=x_{d_A(v)+1}\prod_{w\in V(G) \atop w\neq v}x_{d_A(w)}.$$
Then
$$\sum_{A\subseteq E(G)}\sum_{v\in V(G)}\sum_{e\notin A \atop e\ni v}T_{+}(A,v,e)=\sum_{k=0}^d(d-k)x_{k+1}\frac{d}{dx_k}F_G(x_0,\dots ,x_d).$$
Similarly, let $e=(u,v)$ be an edge of $G$, and let $A'$ be a subset of the edges of $G$ such that $e\in A'$. Let us introduce
$$T_{-}(A',u,e)=x_{d_A(u)-1}\prod_{w\in V(G) \atop w\neq u}x_{d_A(w)}.$$
Then
$$\sum_{A'\subseteq E(G)}\sum_{u\in V(G)}\sum_{e\in A' \atop e\ni u}T_{-}(A,u,e)=\sum_{k=0}^dkx_{k-1}\frac{d}{dx_k}F_G(x_0,\dots ,x_d).$$
Next observe that $T_{+}(A,v,e)=T_{-}(A+e,u,e)$ by definition. Hence
$$\sum_{k=0}^d((d-k)x_{k+1}-kx_{k-1})\frac{d}{dx_k}F_G(x_0,\dots ,x_d)=0.$$
\end{proof}

\begin{Rem}
Note that $G$ does not necessarily be simple. For instance, for $d=4$, then a vertex with two loops shows that $x_0+2x_2+x_4$ is in the ring of coefficients.
\end{Rem}

Next we study the generators of $R^{\partial}$. Let $R_k$ be the set of homogeneous polynomials of degree $k$. This is a vector space on which $\partial $ acts as  a linear transformation. 

\begin{Lemma} \label{eigenvalue}
The eigenvalues $\partial$ on $R_1$  is $\lambda_k=(d-2k)i$ for $k=0,\dots ,d$. (Here $i=\sqrt{-1}$.)
\end{Lemma}

\begin{proof} See Lemma~\ref{eigenvector} for a more precise statement also giving the eigenvectors.

\end{proof}

Let $Q_k$ be the eigenvectors of $\partial$ belonging to $\lambda_k$, that is, $\partial(Q_k)=\lambda_kQ_k$. Since  the $\lambda_k$'s are different, the polynomials $Q_k$ induces $R_1$ as a vector space. In particular, each $x_k$ $k=0,\dots ,d$ can be written as a linear combinations of $Q_k$. This means that we can also write
$$\prod_{j=0}^dx_j^{\alpha_j}=\sum_{\underline{\beta}}c_{\underline{\alpha},\underline{\beta}}\prod_{j=0}^dQ_j^{\beta_j},$$
where the sum runs over the vectors $\underline{\beta}=(\beta_0,\dots ,\beta_d)$ for which $\sum_{j=0}^d\beta_j=\sum_{j=0}^d\alpha_j$. This means that 
$$R_k=\langle \prod_{j=0}^dQ_j^{\beta_j}\ |\ \sum_{j=0}^d\beta_j=k\rangle.$$
Furthermore, since $\dim(R_k)=|\{(\beta_0,\dots ,\beta_d)\ |\ \sum_{j=0}^d\beta_j=k\}|$ the set 
$$T_k=\left\{\prod_{j=0}^dQ_j^{\beta_j}\ |\ \sum_{j=0}^d\beta_j=k\right\}$$ 
is a basis of $R_k$. Note that if $\partial P_1=\mu_1 P_1$ and $\partial P_2=\mu_2 P_2$, then
$$\partial (P_1P_2)=\partial P_1\cdot P_2+P_1\cdot \partial P_2=(\mu_1+\mu_2)P_1P_2.$$
In particular,
$$\partial\left(\prod_{j=0}^dQ_j^{\beta_j}\right)=\left(\sum_{j=0}^d\beta_j\lambda_j\right)\prod_{j=0}^dQ_j^{\beta_j}.$$
Hence,
$$\partial \left(\sum_{\underline{\beta}}c_{\underline{\beta}}\prod_{j=0}^dQ_j^{\beta_j}\right)=\sum_{\underline{\beta}}c_{\underline{\beta}}\left(\sum_{j=0}^d\beta_j\lambda_j\right)\prod_{j=0}^dQ_j^{\beta_j}.$$
Since $T_k$ is a basis, this expression is equal to $0$ if and only if the coefficient $c_{\underline{\beta}}=0$ whenever $\sum_{j=0}^d\beta_j\lambda_j\neq 0$. Hence
$$R_k\cap R^{\partial}=\langle \prod_{j=0}^dQ_j^{\beta_j}\ |\ \sum_{j=0}^d\beta_j=k,\ \sum_{j=0}^d\beta_j\lambda_j=0\rangle.$$
Let
$$S=\left\{(\beta_0,\dots ,\beta_d)\ |\ \beta_j\in \mathbb{Z}_{\geq 0}\ \ (j=0,\dots ,d),\  \sum_{j=0}^d\beta_j\lambda_j=0\right\}.$$
Let us call a vector $\underline{\beta}\in S$ irreducible if there is no $(t_0,\dots ,t_d)\in S$ such that $t_k\leq \beta_k$ for all $k$ and $(t_0,\dots ,t_d)\neq (\beta_0,\dots ,\beta_d)$. 
Let 
$$S_0=\left\{(\beta_0,\dots ,\beta_d)\ |\ (\beta_0,\dots ,\beta_d)\ \mbox{is irreducible}\right\}.$$
This is a finite set since all $\beta_i\geq 0$ and if $(\beta_0,\dots ,\beta_d)$ and $(\beta'_0,\dots ,\beta'_d)$ are irreducible elements, then there must be coordinates $i$ and $j$ such that $\beta_i>\beta'_i$ and 
$\beta_j<\beta'_j$. It is folklore and can be proven by induction on $d$ that such sets are always finite. (Though there is no universal upper bound on the size of the set even for $d=2$.)
Clearly the set
$$\left\{\prod_{j=0}^dQ_j^{\beta_j}\ |\ \underline{\beta}\in S_0\right\}$$
generates $R^{\partial}$ as a ring. 

\begin{Ex}
Let $d=3$ and $\partial Q_{(3)}=3i\cdot Q_{(3)}$, $\partial Q_{(1)}=i\cdot Q_{(1)}$, $\partial Q_{(-1)}=-i\cdot Q_{(-1)}$ and $\partial Q_{(-3)}=-3i\cdot Q_{(-3)}$. Then the generators of the ring $R^{\partial}$ is $Q_{(3)}Q_{(-3)},Q_{(1)}Q_{(-1)},Q_{(3)}Q_{(-1)}^3$ and $Q_{(1)}^3Q_{(-3)}$. Here 
\begin{align*}
Q_{(3)}&=x_0-3i x_1-3x_2+ix_3,\\
Q_{(1)}&=x_0-i x_1+x_2-ix_3,\\
Q_{(-1)}&=x_0+i x_1+x_2+ix_3,\\
Q_{(-3)}&=x_0+3i x_1-3x_2-ix_3.
\end{align*}
\end{Ex}

\begin{Lemma} \label{eigenvector}
Let $L$ be the the linear operator for which in case of $P(z)=\sum_{k=0}^dc_kz^k$ we have $LP=\sum_{k=0}^dc_kx_k$. Let 
$$P_r(z)=(1+iz)^r(1-iz)^{d-r}.$$
Then $LP_r=Q_{(d-2r)}$, in other words, $\partial (LP_r)=(d-2r)i (LP_r)$.
\end{Lemma}

\begin{proof}
For a polynomial $P(z)=\sum_{k=0}^dc_kz^k$ let
$$Q(z)=\sum_{k=0}^dc_k((d-k)z^{k+1}-kz^{k-1}).$$
Then $LQ(z)=\partial (LP)$. Naturally, 
$$Q(z)=dzP(z)-(1+z^2)P'(z).$$
So we can introduce the linear operator $T$ for which $TP=dzP(z)-(1+z^2)P'(z)$. So it is sufficient to prove that $TP_r=(d-2r)i\cdot P_r$. This is indeed true:
\begin{align*}
TP_r&=dzP_r(z)-(1+z^2)P_r'(z)\\
    &=dz(1+iz)^r(1-iz)^{d-r}-(1+z^2)(ri(1+iz)^{r-1}(1-iz)^{d-r}-(d-r)i(1+iz)^r(1-iz)^{d-r-1})\\
		&=(1+iz)^{r-1}(1-iz)^{d-r-1}(dz(1+z^2)-(1+z^2)(ri(1-iz)-(d-r)i(1+iz)))\\
		&=(1+iz)^{r-1}(1-iz)^{d-r-1}(dz(1+z^2)-(1+z^2)(dz-(d-2r)i))\\
		&=(1+iz)^{r-1}(1-iz)^{d-r-1}((d-2r)i(1+z^2))\\
		&=(d-2r)i\cdot (1+iz)^{r}(1-iz)^{d-r}\\
		&=(d-2r)i P_r(z).
\end{align*}
This completes the proof.
\end{proof}

\section{Duality theorem} \label{duality}

In this section we establish a connection between summing to subgraphs and summing to orientations. The main theorem of this section is the following.

\begin{Def}
Given an orientation $\mathcal{O}$ of the edges, the oriented degree $d_{\mathcal{O}}(v)$ is the out-degree minus the in-degree of the vertex $v$. (This is a number between $d$ and $-d$ having the same parity as $d$. The sum of the oriented degrees is always $0$.)
\end{Def}

\begin{Th} \label{orientations}
Let $G$ be a $d$--regular graph, and let us normalize $Q_{(k)}(x_0,\dots ,x_d)$ in such a way that the coefficient of $x_0$ is $1$ and it belongs to the eigenvalue $ki$, that is, $\partial Q_{(k)}=ki\cdot Q_{(k)}$. 
For any graph $G$ we have
$$F_G(x_0,\dots ,x_d)=\frac{1}{2^{e(G)}}\sum_{\mathcal{O}}\prod_{v\in V}Q_{(d_{\mathcal{O}}(v))}$$
\end{Th}

\begin{Lemma} \label{orientation_lemma}
Let $G_1:=\left( \begin{array}{cc} 1 & -i \\ 1 & i \end{array}\right)$, and $G_{-1}:=\left( \begin{array}{cc} 1 & i \\ 1 & -i \end{array}\right)$, where $i=\sqrt{-1}$ and the rows and columns are labelled by $0$ and $1$. Suppose that $f(\sigma_1,\dots ,\sigma_d)=x_r$ if $\sum_{k=1}^d\sigma_k=r$. For $\underline{\gamma}=(\gamma_1,\dots ,\gamma_d)\in \{-1,1\}^d$ and $\underline{\tau}=(\tau_1,\dots,\tau_d)\in \{0,1\}^d$ let
$$\widehat{f}_{G_{\underline{\gamma}}}(\tau_1,\dots,\tau_d)=\sum_{(\sigma_1,\dots ,\sigma_d)\in \{0,1\}^d}\prod_{k=1}^dG_{\gamma_k}(\tau_k,\sigma_k)\cdot f(\sigma_1,\dots,\sigma_d).$$
Then
$$\widehat{f}_{G_{\underline{\gamma}}}(\tau_1,\dots,\tau_d)=Q_{M(\underline{\gamma},\underline{\tau})}(x_0,\dots ,x_d),$$
where $M(\underline{\gamma},\underline{\tau})=-\sum_{k=1}^d\gamma_k(2\tau_k-1)$.
\end{Lemma}

\begin{proof}
Let $\overline{\tau}_k=2\tau_k-1$. Note that $\overline{\tau}_k\in \{-1,1\}$. Next observe that
$$G_{\gamma_k}(\tau_k,\sigma_k)=i^{\sigma_k\gamma_k\overline{\tau}_k}.$$
Let $z$ be a new variable and $L$ be a linear operator such that $L(z^k)=x_k$ for $k=0,\dots ,d$.  Then
$$\widehat{f}_{G_{\underline{\gamma}}}(\tau_1,\dots,\tau_d)=\sum_{(\sigma_1,\dots ,\sigma_d)\in \{0,1\}^d}\prod_{k=1}^di^{\sigma_k\gamma_k\overline{\tau}_k}L(z^{\sigma_1+\dots+\sigma_d})=L\left(\sum_{(\sigma_1,\dots ,\sigma_d)\in \{0,1\}^d}\prod_{k=1}^d(i^{\sigma_k\gamma_k\overline{\tau}_k}z^{\sigma_k})\right)=$$
$$=L\left(\prod_{k=1}^d(1+i^{\gamma_k\overline{\tau}_k}z)\right)=L\left(\prod_{k=1}^d(1+i\gamma_k\overline{\tau}_kz)\right)=Q_{M(\underline{\gamma},\underline{\tau})}(x_0,\dots ,x_d).$$
We use that if $s\in \{-1,1\}$, then $i^s=si$, and in the last step we used Lemma~\ref{eigenvector}.

\end{proof}

Now we are ready to prove Theorem~\ref{orientations}. It is just a simple application of gauge transformations.

\begin{proof}[Proof of Theorem~\ref{orientations}] 
Observe that $G_1G_{-1}^T=2\mathrm{Id}$. Let us fix an orientation $\mathcal{O}$ of the edges. To an oriented edges $(u,v)$ let $G_{uv}=\frac{1}{\sqrt{2}}G_1$ and $G_{vu}=\frac{1}{\sqrt{2}}G_{-1}$. For each vertex $v$ this gives a vector $\underline{\gamma}^v\in \{-1,1\}^d$: each edge oriented outward gives $1$, and each edge oriented inward gives a $-1$.

For any other orientation $\mathcal{O}'$ we can consider the set of edges, where $\mathcal{O}$ and $\mathcal{O}'$ gives different orientation of the edge. Identify this set with $\{e\ |\ \tau_e=1\}$. By the gauge transformation theorem with $\mathcal{X}=\mathcal{Y}=\{0,1\}$ and gauges above, we have
$$\sum_{\underline{\sigma}\in \mathcal{X}^E}\prod_{v\in V}f_v(\underline{\sigma}_{\partial v})=\sum_{\underline{\tau}\in \mathcal{Y}^E}\prod_{v\in V}\widehat{f_v}(\underline{\tau}_{\partial v}).$$
The left hand side is clearly $F_G(x_0,\dots ,x_d)$. Using Lemma~\ref{orientation_lemma} we know that the right hand side is
$$\frac{1}{2^{e(G)}}\sum_{\underline{\tau}\in \mathcal{Y}^E}\prod_{v\in V}Q_{M(\underline{\gamma}^v,\underline{\tau}_{\partial v})}(x_0,\dots ,x_d).$$
The last observation is that $M(\underline{\gamma}^v,\underline{\tau}_{\partial v})=d_{\mathcal{O}'}(v)$ for every vertex $v$ and orientation $\mathcal{O}'$. Indeed, if $\underline{\gamma}^v=(1,1,\dots ,1)$ and $\underline{\tau}=(0,\dots ,0)$, then every edge has outward orientation in $\mathcal{O}$, and $\mathcal{O}'$ agrees with $\mathcal{O}$. Then
$$M(\underline{\gamma}^v,\underline{\tau}_{\partial v})=-\sum_{k=1}^d1\cdot (2\cdot 0-1)=d=d_{\mathcal{O}'}(v).$$
Now it is easy to check that after every change in $\underline{\gamma}^v$ and $\underline{\tau}_{\partial v}$ the same change occurs in the left and right hand side.

\end{proof}

\section{Eulerian orientations and half-graphs} \label{Eulerian orientations}

In this section $d$ is even. The main theorem of this section is the following.

\begin{Th} \label{Euler orientations, regular}
Let $\underline{s}=(s_0,s_1,\dots ,s_d)$ be defined as follows.
$$s_k=\left\{ \begin{array}{cc} \frac{\binom{d}{d/2}\binom{d/2}{k/2}}{2^{d/2}\binom{d}{k}} & \mbox{if}\ \ k\ \ \mbox{is even}, \\
                                 0  & \mbox{if}\ \ k\ \ \mbox{is odd}.
																\end{array} \right.$$
Then $F_G(s_0,\dots ,s_d)$ counts the number of Eulerian orientations of a $d$--regular graph $G$.															
\end{Th}

\begin{Ex}
For a $4$--regular graph $F_G(\frac{3}{2},0,\frac{1}{2},0,\frac{3}{2})$ counts the number of Eulerian orientations. For a $6$--regular graph $F_G(\frac{20}{8},0,\frac{5}{8},0,\frac{5}{8},0, \frac{20}{8})$, for an $8$--regular graph \\ $F_G(\frac{70}{16},0,\frac{14}{16},0,\frac{42}{80},0, \frac{14}{16},0,\frac{70}{16})$ counts the number of Eulerian orientations.
\end{Ex}

The non-regular version is exactly what one would expect.

\begin{Th} \label{Euler orientations, non-regular}
For an even $d$ let $\underline{s}^{(d)}=(s^{(d)}_0,s^{(d)}_1,\dots ,s^{(d)}_d)$ be defined as follows.
$$s^{(d)}_k=\left\{ \begin{array}{cc} \frac{\binom{d}{d/2}\binom{d/2}{k/2}}{2^{d/2}\binom{d}{k}} & \mbox{if}\ \ k\ \ \mbox{is even}, \\
                                 0  & \mbox{if}\ \ k\ \ \mbox{is odd}.
																\end{array} \right.$$
Let $\underline{s}$ be the vector that we get if we substitute $\underline{s}^{(d)}$ into $x^{v}_0,\dots ,x^{v}_{d_v}$ if $d_v=d$.
Then $F_G(\underline{s})$ counts the number of Eulerian orientations of a graph $G$.															
\end{Th}

Before we start to prove Theorem~\ref{Euler orientations, regular} and Corollary~\ref{Euler orientations, non-regular} we give the corresponding statement for the number of half-graphs.

\begin{Th} \label{half-graphs, regular}
Let $\underline{c}=(c_0,c_1,\dots ,c_d)$ be defined as follows.
$$c_k=\left\{ \begin{array}{cc} (-1)^{k/2}\frac{\binom{d}{d/2}\binom{d/2}{k/2}}{2^{d/2}\binom{d}{k}} & \mbox{if}\ \ k\ \ \mbox{is even}, \\
                                 0  & \mbox{if}\ \ k\ \ \mbox{is odd}.
																\end{array} \right.$$
Then $F_G(c_0,\dots ,c_d)$ counts the number of half-graphs of a $d$--regular graph $G$.															
\end{Th}

\begin{Ex}
For a $4$--regular graph $F_G(\frac{3}{2},0,-\frac{1}{2},0,\frac{3}{2})$ counts the number of half-graphs. For a $6$--regular graph $F_G(\frac{20}{8},0,-\frac{5}{8},0,\frac{5}{8},0, -\frac{20}{8})$, for an $8$--regular graph \\ $F_G(\frac{70}{16},0,-\frac{14}{16},0,\frac{42}{80},0, -\frac{14}{16},0,\frac{70}{16})$ counts the number of half-graphs.
\end{Ex}

The non-regular version is exactly what one would expect.

\begin{Th} \label{half-graphs, non-regular}
For an even $d$ let $\underline{c}^{(d)}=(c^{(d)}_0,c^{(d)}_1,\dots ,c^{(d)}_d)$ be defined as follows.
$$c^{(d)}_k=\left\{ \begin{array}{cc} (-1)^{k/2}\frac{\binom{d}{d/2}\binom{d/2}{k/2}}{2^{d/2}\binom{d}{k}} & \mbox{if}\ \ k\ \ \mbox{is even}, \\
                                 0  & \mbox{if}\ \ k\ \ \mbox{is odd}.
																\end{array} \right.$$
Let $\underline{c}$ be the vector that we get if we substitute $\underline{c}^{(d)}$ into $x^{v}_0,\dots ,x^{v}_{d_v}$ if $d_v=d$.
Then $F_G(\underline{c})$ counts the number of half-graphs of a graph $G$.															
\end{Th}

Before we start to prove the above theorems we collect some simple observations in a lemma.

\begin{Lemma} Let $A^{(d)}$ be the matrix of size $(d+1)\times (d+1)$ with rows and columns labelled by $0,1,\dots ,d$ such that $A^{(d)}_{k,k+1}=d-k$ for $k=0,\dots ,d$ and $A^{(d)}_{k,k-1}=-k$ for $k=1,\dots ,d$. Then $\underline{b}=(b_0,\dots ,b_d)$ satisfies that $\underline{b} A^{(d)}=\lambda \underline{b}$ for some $\lambda$ if and only if the polynomial $Q_{\underline{b}}(x_0,\dots ,x_d)=\sum_{k=0}^db_kx_k$ satisfies that $\partial Q=\lambda Q$. Furthermore, if $\underline{c}=(c_0,\dots,c_d)$ such that $A^{(d)}\underline{c}=\mu \underline{c}$ for some $\mu\neq \lambda$, then $Q_{\underline{b}}(c_0,\dots ,c_d)=0$.
\end{Lemma}

\begin{proof} The claim that $\underline{b}=(b_0,\dots ,b_d)$ satisfies that $\underline{b} A^{(d)}=\lambda \underline{b}$ for some $\lambda$ if and only if the polynomial $Q_{\underline{b}}(x_0,\dots ,x_d)=\sum_{k=0}^db_kx_k$ satisfies that $\partial Q=\lambda Q$ is practically a tautology. The second statement that if $\underline{c}=(c_0,\dots,c_d)$ such that $A^{(d)}\underline{c}=\mu \underline{c}$ for some $\mu\neq \lambda$, then $Q(c_0,\dots ,c_d)=0$ follows from the following argument: $\lambda (\underline{b},\underline{c})=\underline{b} A^{(d)}\underline{c}=\mu (\underline{b},\underline{c})$ implies that 
$(\underline{b},\underline{c})=0$ since $\lambda\neq \mu$. This is equivalent with $Q(c_0,\dots ,c_d)=0$.

\end{proof}

\begin{Ex}
$$A^{(3)}=\left(\begin{array}{cccc}
0 & 3 & 0 & 0 \\
-1 & 0 & 2 & 0 \\
0 & -2 & 0 & 1 \\
0 & 0 & -3 & 0 \\
\end{array}
\right).$$

\end{Ex}

\begin{Rem} If we delete the negative signs in the matrix $A^{(d)}$ the obtained matrix is called the Clement-matrix or Sylvester-Katz matrix. Its eigenvalues are $d,d-2,\dots ,-d$.

\end{Rem}

\begin{proof}[Proof of Theorems~\ref{Euler orientations, regular} and \ref{Euler orientations, non-regular}]
We only prove the regular case. The proof of the non-regular case is essentially the same.
The proof consists of the following steps.
First we show that the vector $(s_0,s_1,\dots ,s_d)$ is the right eigenvector of the matrix $A^{(d)}$ belonging to the eigenvalue $0$. From this and the lemma it follows that $Q_{(j)}(s_0,s_1,\dots ,s_d)=0$ if $j\neq 0$. 
From Theorem~\ref{orientations} we know that
$$F_G(x_0,\dots ,x_d)=\frac{1}{2^{e(G)}}\sum_{\mathcal{O}}\prod_{v\in V}Q_{(d_{\mathcal{O}}(v))}.$$
So evaluating at $\underline{s}$, most of the terms vanish and only the Eulerian orientations remain:
$$F_G(s_0,s_1,\dots ,s_d)=\frac{c_{\underline{0}}}{2^{e(G)}}Q_0(s_0,s_1,\dots ,s_d)^n,$$
where $c_{\underline{0}}$ is the number of Eulerian orientations. Finally, we show that $\underline{s}$ is normalized is such a way that 
$$\frac{1}{2^{e(G)}}Q_0(s_0,s_1,\dots ,s_d)^n=1,$$
and so $F_G(s_0,s_1,\dots ,s_d)=c_{\underline{0}}$.
\bigskip

One can check directly that $(s_0,s_1,\dots ,s_d)$ is the right eigenvector of the matrix $A^{(d)}$ belonging to the eigenvalue $0$. Alternatively, let $\underline{s'}=(s'_0,s'_1,\dots ,s'_d)$ and  $s'_0=1$. Using the equation $A^{(d)}\underline{s'}=\underline{0}$, equivalently equations $(d-k)s'_{k+1}-ks'_{k-1}=0$, we get that
$s'_k=0$ if $k$ is odd, and $s'_0=1$ implies $s'_2=\frac{1}{d-1}$, $s'_4=\frac{1}{d-1}\cdot \frac{3}{d-3}$ and in general
$$s'_{2t}=\frac{1\cdot 3\cdot \dots \cdot (2t-1)}{(d-1)\cdot (d-3)\cdot \dots \cdot (d-(2t-1))}.$$
Using that $(2t)!!:=1\cdot 3\cdot \dots \cdot (2t-1)=\frac{(2t)!}{2^tt!}$ we can further simplify it:
$$s'_{2t}=\frac{(2t)!!(d-2t)!!}{d!!}=\frac{\frac{(2t)!}{2^tt!}\frac{(d-2t)!}{2^{d/2-t}(d/2-t)!}}{\frac{d!}{2^{d/2}(d/2)!}}=\frac{\binom{d/2}{t}}{\binom{d}{2t}}.$$
Note that $s_{t}=\frac{\binom{d}{d/2}}{2^{d/2}}s'_t$. Next we evaluate $Q_{(0)}(\underline{s'})$.  We have a general formula for $Q_{(k)}$ which is particularly simple in case of $k=0$, namely, from Lemma~\ref{eigenvector} we have
$$Q_{(0)}(a_0,\dots ,a_d)=L((1+iz)^{d/2}(1-iz)^{d/2})=L((1+z^2)^{d/2})=\sum_{t=0}^{d/2}\binom{d/2}{t}a_{2t}.$$
Hence
$$Q_{(0)}(\underline{s'})=\sum_{t=0}^{d/2}\binom{d/2}{t}\frac{\binom{d/2}{t}}{\binom{d}{2t}}.$$
Observe that
$$\sum_{t=0}^{d/2}\binom{d/2}{t}\frac{\binom{d/2}{t}}{\binom{d}{2t}}=\frac{1}{\binom{d}{d/2}}\sum_{t=0}^{d/2}\binom{2t}{t}\binom{2(d/2-t)}{d/2-t}=\frac{1}{\binom{d}{d/2}}4^{d/2}=\frac{2^d}{\binom{d}{d/2}}.$$
Hence
$$F_G(s'_0,s_1,\dots ,s'_d)=\frac{c_{\underline{0}}}{2^{e(G)}}Q_0(s'_0,s_1,\dots ,s'_d)^n=\frac{2^{dn/2}}{{\binom{d}{d/2}}^n}c_{\underline{0}},$$
whence
$$F_G(s_0,s_1,\dots ,s_d)=c_{\underline{0}}.$$
\end{proof}

\begin{proof}[Proof of Theorems~\ref{half-graphs, regular} and \ref{half-graphs, non-regular}]
Again we only prove the regular case, the proof of the non-regular case is essentially the same.
For an even $d$ let $\textbf{e}_{d/2}$ be the vector $(0,\dots ,0,1,0,\dots ,0)$ with a $1$ at $d/2$-th coordinate. Then
$F_G(\textbf{e}_{d/2})$ is the number of $(d/2)$--regular subgraphs. We know that $F_G(\textbf{e}_{d/2})=F_G(\textbf{R}_{\pi/4}\textbf{e}_{d/2})$, so it is enough to show that the vector $\underline{c}=(c_0,c_1,\dots ,c_d)$ is exactly $\textbf{R}_{\pi/4}\textbf{e}_{d/2}$. In other words, we need to show that the $(d/2)$-th column vector of $\textbf{R}_{\pi/4}$ is $\underline{c}$. 

By putting together Lemma~\ref{rotation matrix rows} with the definition of the rotation matrix $\textbf{R}_{t}$ we get that the $r$-th element $(r=0,1,\dots ,d)$ of the $k$-th row is the coefficient of $x^r$ in the polynomial $(x\cos(t)-\sin(t))^k(x\sin(t)+\cos(t))^{d-k}$. So we need the the coefficient of $x^{d/2}$ for $t=\pi/4$. Then
$$(x\cos(\pi/4)-\sin(\pi/4))^k(x\sin(\pi/4)+\cos(\pi\/4))^{d-k}=\frac{1}{2^{d/2}}(x-1)^k(x+1)^{d-k}.$$
The coefficient of $x^{d/2}$ in $(x-1)^k(x+1)^{d-k}$ is
\begin{align*}
\sum_{j=0}^{d/2}(-1)^{k-j}\binom{k}{j}\binom{d-k}{d/2-j}&=\sum_{j=0}^{d/2}(-1)^{k-j}\frac{k!(d-k)!}{j!(k-j)!(d/2-j)!(d/2-k+j)!}\\
&=(-1)^k\frac{\binom{d}{d/2}}{\binom{d}{k}}\sum_{j=0}^{d/2}(-1)^j\binom{d/2}{j}\binom{d/2}{k-j}.
\end{align*}
Here $\sum_{j=0}^{d/2}(-1)^j\binom{d/2}{j}\binom{d/2}{k-j}$ is also the coefficent of $x^k$ in \\ $(1-x)^{d/2}(1+x)^{d/2}=(1-x^2)^{d/2}$ which is clearly $0$ if $k$ is odd, and $(-1)^{k/2}\binom{d/2}{k/2}$ if $k$ is even. So the coefficient of $x^{d/2}$ in $\frac{1}{2^{d/2}}(x-1)^k(x+1)^{d-k}$ is exactly the $c_k$ defined in Theorem~\ref{half-graphs, regular}.

\end{proof}

Now we are ready to prove Theorems~\ref{Schrijver-non-regular} and Corollary~\ref{Schrijver-regular}.

\begin{proof}[Proof of Theorem~\ref{Schrijver-non-regular} and Corollary~\ref{Schrijver-regular}]
In case of a $d$--regular graph we have
$$\varepsilon(G)=F_G(s_0,s_1,\dots ,s_d)\geq s_0^n+s_d^n=2 \left(\frac{\binom{d}{d/2}}{2^{d/2}}\right)^n.$$
In case of non-regular graphs, the only difference is that we have to substitute \\ $(s^{(d_v)}_0,s_1^{(d_v)},\dots ,s^{(d_v)}_{d_v})$ into $x^v_0,\dots ,x^v_{d_v}$ in $F_G(\underline{x})$.

\end{proof}

We can also prove Theorem~\ref{orientations and half-graphs, non-regular} and Corollary~\ref{orientations and half-graphs, regular}.

\begin{proof}[Proof of Theorem~\ref{orientations and half-graphs, non-regular} and Corollary~\ref{orientations and half-graphs, regular}]
This is clear from the fact that $|c_k|=s_k$ and 
$$\varepsilon(G)=F_G(s_0,\dots,s_d)\geq F_G(c_0,\dots,c_d)=h(G)$$
for a $d$--regular graph $G$ and similarly, $\varepsilon(G)=F_G(\underline{s})\geq F_G(\underline{c})=h(G)$ for a non-regular graph $G$. 
If $G$ is non-bipartite, then it contains and odd cycle and the contribution of this odd cycle to the sums shows that there cannot be equality.
It is also clear that if $G$ is a bipartite graph, then there is a bijection between half-graphs and the oriented edges going from one part to the other of the bipartite graph.
\end{proof}

\begin{Rem}
The vector $(s_0,s_1,\dots ,s_d)$ or equivalently $(s'_0,s'_1,\dots ,s'_d)$ has another specialty: the functions $\widehat{a_r}(t)$ are constant. Indeed,
$$\widehat{a_0}(t)=\sum_{k=0}^ds'_k\binom{d}{k}\cos(t)^k\sin(t)^{d-k}=\sum_{k=0}^{d/2}\binom{d/2}{k}(\cos(t)^2)^{k}(\sin(t)^2)^{d/2-k}=(\cos(t)^2+\sin(t)^2)^{d/2}=1.$$
Then using the formulas $\frac{d}{dt}\widehat{a_k}(t)=(d-k)\widehat{a_{k+1}}(t)-k\widehat{a_{k-1}}(t),$
we get that the other $\widehat{a_k}(t)$ functions are constant too.
\end{Rem}

\section{Orientations of $3$--regular graphs}

In this section we prove Theorem~\ref{sources and sinks}.

\begin{proof}[Proof of Theorem~\ref{sources and sinks}]
Let $H_G(y_{-3},y_{-1},y_1,y_3)=\sum_{\mathcal{O}}\prod_{v\in V}y_{d_{\mathcal{O}}(v)},$
where $d_{\mathcal{O}}(v)$ is the oriented degree of the vertex $v$. Then
$$H_G\left(\frac{1}{t},1,1,t\right)=\sum_{\mathcal{O}}t^{n_+(\mathcal{O})-n_-(\mathcal{O})}.$$
We know that 
$$\frac{1}{2^{3n/2}}H_G(y_{-3},y_{-1},y_1,y_3)=F_G(x_0,x_1,x_2,x_3)$$
for some $x_0,x_1,x_2,x_3$. In fact, we will show that there are $a$ and $b$ such that
$$\frac{1}{2^{3n/2}}H_G(y_{-3},y_{-1},y_1,y_3)=F_G(a,0,0,b)$$
and
$$a=\frac{1}{2}(t^{1/4}+t^{-1/4})\ \  \mbox{and}\ \  b=\frac{i}{2}(t^{1/4}-t^{-1/4}).$$ 
We know that
$$F_G(x_0,x_1,x_2,x_3)=\frac{1}{2^{e(G)}}H_G(Q_{(-3)},Q_{(-1)},Q_{(1)},Q_{(3)}),$$
and that $Q_{(3)}Q_{(-3)},Q_{(1)}Q_{(-1)},Q_{(3)}Q_{(-1)}^3+Q_{(-3)}Q_{(1)}^3$ generate $R^{\partial}$. This means that if we choose any $x_0,x_1,x_2,x_3$ for which $Q_{(3)}Q_{(-3)}=t\cdot 1/t=1,Q_{(1)}Q_{(-1)}=1$ and $Q_{(3)}Q_{(-1)}^3+Q_{(-3)}Q_{(1)}^3=t+1/t$, then we get the same $F_G(x_0,x_1,x_2,x_3)$. Note that
\begin{align*}
Q_{(3)}Q_{(-3)}=&x_0^2-6x_0x_2+9x_1^2-6x_1x_3+9x_2^2+x_3^2,\\
Q_{(3)}Q_{(-3)}=&x_0^2+2x_0x_2+x_1^2+2x_1x_3+x_2^2+x_3^2,\\
Q_{(3)}Q_{(-1)}^3+Q_{(-3)}Q_{(-1)}^3=&-6x_1^4-6x_2^4+2x_0^4-12x_1^2x_3^2+2x_3^4+48x_0x_1x_2x_3+\\
&12x_0^2x_1^2-12x_0^2x_2^2+48x_1x_2^2x_3+48x_0x_1^2x_2-16x_0x_2^3\\
&-12x_0^2x_3^2-16x_1^3x3+36x_1^2x_2^2+12x_2^2x_3^2\\
\end{align*}
The reason why it is enough to check $Q_{(3)}Q_{(-1)}^3+Q_{(-3)}Q_{(-1)}^3$ instead of $Q_{(3)}Q_{(-1)}^3$ and $Q_{(-3)}Q_{(-1)}^3$ is that these are conjugate pairs and $F_G(x_0,x_1,x_2,x_3)$ has real coefficients.
So if we choose $(x_0,x_1,x_2,x_3)=(a,0,0,b)$, then the first two equations reduce to $a^2+b^2=1$, the third one to $2a^4+2b^4-12a^2b^2=t+1/t$. It is easy to check that  the above $a$ and $b$ indeed satisfy these equations.          
\medskip

Since $G$ is connected, every subgraph different from the empty and the complete graph has a vertex of degree $1$ or $2$. Hence
$F_G(a,0,0,b)=a^n+b^n$. Hence
$$\frac{1}{2^{3n/2}}H_G\left(\frac{1}{t},1,1,t\right)=\left(\frac{1}{2}(t^{1/4}+t^{-1/4})\right)^n+\left(\frac{i}{2}(t^{1/4}-t^{-1/4})\right)^n.$$
Then
$$\frac{1}{2^{n/2-1}}H_G\left(\frac{1}{t},1,1,t\right)=(t^{1/4}+t^{-1/4})^n+(i(t^{1/4}-t^{-1/4}))^n.$$
By binomial theorem we have
\begin{align*}
(t^{1/4}+t^{-1/4})^n+(i(t^{1/4}-t^{-1/4}))^n&=\sum_{k=0}^n\binom{n}{k}(1+i^n(-1)^{n-k})t^{k/4-(n-k)/4}\\
&=\sum_{r=-n/2}^{n/2}\binom{n}{n/2+r}(1+i^n(-1)^{n/2+r})t^{-r/2}.
\end{align*}
Now observe that $1+i^n(-1)^{n/2+r}=1+(-1)^{n+r}=1+(-1)^r$, and so
$$(t^{1/4}+t^{-1/4})^n+(i(t^{1/4}-t^{-1/4}))^n=2\sum_{s=-\lfloor n/4\rfloor}^{\lfloor n/4 \rfloor}\binom{n}{n/2+2s}t^{-s}.$$
Hence
$$\sum \mathbb{P}(n_+(\mathcal{O})-n_-(\mathcal{O})=k)t^k=\frac{1}{2^{3n/2}}H_G(1/t,1,1,t)=\sum_{s=-\lfloor n/4\rfloor}^{\lfloor n/4 \rfloor}\frac{\binom{n}{n/2+2s}}{2^{n-1}}t^{-s}.$$
Now comparing the coefficent of $t^k$ we get the claim.

\end{proof}

\noindent \textbf{Acknowledgment.} The second author is very grateful to M. Chertkov for useful discussion and help with references.

\bibliographystyle{siamnodash}

\bibliography{counting_subgraphs_orientations}

\end{document}